\documentclass[12pt]{amsart}
\usepackage{amsmath,amssymb}
\usepackage{rotating}
\usepackage{graphicx}
\pdfoutput=1

\newcommand{\tre}{\text{Re}}

\newcommand{\hg}{\hat{g}}

\newcommand{\intii}{\int_{-\infty}^\infty}

\newcommand{\supp}{{\rm supp}}

\theoremstyle{definition}
\newtheorem{theorem}{Theorem}

\newtheorem*{millertheorem}{Refined 1-Level Density}

\newtheorem*{lemma}{Lemma}

\begin{document}
\title{The quadratic character experiment}
\author{Jeffrey Stopple}
\email{stopple@math.ucsb.edu}

\begin{abstract}
A fast new algorithm is used compute the zeros of  the quadratic character $L$-functions for all negative fundamental discriminants with absolute value $10^{12}<d<10^{12}+10^7$.  These are compared to the 1-level density, including various lower order terms. These terms come from, on the one hand the Explicit Formula, and on the other the $L$-functions Ratios Conjecture.  The latter give a much better fit to the data, providing numerical evidence for the conjecture.
\end{abstract}
\maketitle

\section{Introduction.}
\subsection*{Predictions} Standard conjectures \cite{KS} predict that the low lying zeros of quadratic Dirichlet $L$-functions should be distributed according to a symplectic random matrix model.  To make this more precise, we'll introduce some notation.  
Let $\chi_d$ be a real, primitive character modulo $d$, and suppose furthermore that $\chi_d(-1)\cdot d$ is a fundamental discriminant.
Let $g(\tau)$ be a Schwartz class test function.   Then the 1-level density for the zeros $1/2+\gamma_d$ of $L(s,\chi_d)$ should satisfy
\begin{multline}\label{Eq:0}
\frac1{X^\ast}\sum_{d \le X} \sum_{\gamma_d}
g\left(\gamma_d \frac{\log X}{2\pi}\right)  =\\
\int_{-\infty}^\infty g(\tau) \left(1 - \frac{\sin(2\pi \tau)}{2\pi
\tau}\right) d\tau + O\left(\frac{1}{\log X}\right),
\end{multline}
where $X^*$ is the cardinality of fundamental discriminants $\chi_d(-1)\cdot d$ with $d<X$.   This is a theorem \cite{OS} if the support of the Fourier transform of $g$ is suitably restricted.

Recently Conrey and Snaith \cite{CS1} made a precise prediction for the lower order arithmetic terms in the 1-level density.
Their prediction is conditional, assuming the $L$-functions Ratios Conjecture \cite{CFZ}.  Miller \cite{Miller} then proved  (under typical restrictions for the test function $g(\tau)$) that these lower order terms exist and agree with the prediction in \cite{CS1}.

\subsection*{Experiments} Zeros of Dirichlet $L$-functions were first computed by Davies and Haselgrove, by Spira, and by Rumley.  Rubinstein, as one portion of his thesis \cite{Rub}, was the first to compute enough low lying zeros to meaningfully test (\ref{Eq:0}).  However, the numerical methods developed in \cite{Rub} were optimized to compute $L(1/2+it,\chi_d)$ for large real  $t$ rather than large $d$.

\subsection*{This paper} The next section develops an algorithm to compute low lying zeros ($0\le t<1.$) which is fast for large $d$.  This is a modification of the idea behind \cite{Stopple}.  The subsequent section has a discussion of the data from the computation of the zeros of approximately $3\cdot 10^6$ quadratic character $L$-functions for negative fundamental discriminants $-d$ with $d>10.^{12}$.  This is followed by some implementation notes, and an Appendix on Miller's \lq refined\rq\ 1-level density and the $L$-functions Ratios Conjecture.

\subsubsection*{Acknowledgements}  Thanks to Mike Rubinstein for sharing his data from \cite{Rub}, and David Farmer for pointing me towards \cite{Miller}  Thanks to both Steven J. Miller and to the anonymous referee for their careful reading of the manuscript and numerous helpful suggestions.

\section{Algorithm}
We are going to compute the $L$-function on the critical line by means of an approximate functional equation, an idea that goes back to Lavrik and was first implemented by Weinberger \cite{W}.  With $\chi$ a real character modulo $d$, let $a=(1-\chi(-1))/2$, and with $t>0$ use $s$ to denote $(1/2+it+a)/2$.\footnote{We are not actually assuming the Generalized Riemann Hypothesis, but we are only looking for zeros on the critical line.}  Define
\begin{equation}\label{Eq:1}
Z(t,\chi)=\xi(1/2+it,\chi)
=\sum_n \chi(n) n^a\,2\, \text{Re}(G(s,\pi n^2/d)),
\end{equation}
where
\[
G(s,x)=x^{-s}\Gamma(s,x)=\int_1^\infty \exp(-yx)y^s \frac{dy}{y}.
\]

As in \cite{W}, the tail of the series, the sum of terms $n>N$, is bounded by $d^2\exp(-N^2\pi/d)/(\pi N)^2$, so if we want to compute to $D$ digits of accuracy, we should have
\begin{equation}\label{Eq:fine}
d^2\exp(-N^2\pi/d)/(\pi N)^2<10^{-D}.
\end{equation} 
Certainly
\begin{equation}\label{Eq:crude}
N\ge d^{1/2}\log(d^210^D)^{1/2}\pi^{-1/2}
\end{equation}
would suffice; later we'll see we can do better given any particular $d$.

Differentiating with respect to $x$ under the integral defining $G(s,x)$ we see that
\begin{equation}\label{Eq:diffintegral}
\frac{d}{dx}G(s,x)=-\int_1^\infty\exp(-xy)y^{s+1}\frac{dy}{y}=-G(s+1,x),
\end{equation}
while integration by parts, on the other hand, gives
\begin{equation}\label{Eq:intbyparts}
G(s+1,x)=\frac{\exp(-x)}{x}+\frac{s}{x}G(s,x).
\end{equation}
Equations (\ref{Eq:diffintegral}) and (\ref{Eq:intbyparts}) give a nice recursive relation  for all the derivatives $G^{(k)}(s,x)$ in terms of $G(s,x)$.    This, in turn, motivates a consideration of Taylor expansions.

Suppose we compute $G(s,x)$ by a Taylor series expansion (in the second variable, centered at $x_0$) to $B$ terms, where $B$ is a parameter to be determined.
\begin{lemma}
We can bound the remainder in the Taylor expansion by a function $R_B(x,x_0)$ (defined below) which satisfies
\begin{equation}\label{Eq:remainder}
R_B(x,x_0)\le\frac{\left(\frac{x}{x_0}-1\right)^B}{B!}\Gamma(B,x_0)\le \frac{\left(\frac{x}{x_0}-1\right)^B}{B}
\end{equation}
\end{lemma}
\begin{proof}
We have
\[
\left|G^{(B)}(s,x)\right|=\left|G(s+B,x)\right|\le \int_1^\infty \exp(-xy)y^B dy,
\]
since $s$ is in the critical strip.  By the integral formula for the remainder in Taylor's theorem, we can bound that remainder by
\begin{align*}
R_B(x,x_0)\overset{\text{def.}}=&\frac{1}{B!}\int_{x_0}^x\int_1^\infty\exp(-uy)y^Bdy\,(x-u)^B du.\\
\intertext{Change the order of integration and let $t=x-u$ to get}
=&\frac{-1}{B!}\int_1^\infty\exp(-xy)\int_{x-x_0}^0\exp(-ty)t^Bdt\,y^B dy.\\
\intertext{Now integrate by parts in the $t$ integral to get}
=&\frac{1}{B!}\int_1^\infty\exp(-xy)(x-x_0)^B\exp((x-x_0)y)y^{B-1}dy \\
&\qquad\qquad\qquad-R_{B-1}(x,x_0).
\end{align*}
Or, in other words,
\begin{align*}
R_{B}(x,x_0)+R_{B-1}(x,x_0)=&\frac{(x-x_0)^B}{B!}G(B,x_0)\\
=&\frac{\left(\frac{x}{x_0}-1\right)^B}{B!}\Gamma(B,x_0).
\end{align*}
This implies the first inequality.  For the second, we observe
\begin{multline*}
\Gamma(B,x_0)=\int_{x_0}^\infty \exp(-y)y^B\frac{dy}{y}\\
\le \int_0^\infty \exp(-y)y^B \frac{dy}{y}=\Gamma(B)=(B-1)!
\end{multline*}
\end{proof}
The first inequality is stronger, so it's good for the actual computation.  The second is weaker, but simple enough to be useful in proving the theorem.

Now we're ready to put the Taylor expansions to good use.  Similar to the method of \cite{Stopple}, we partition the set
$
\{n^2\,|\, 1\le n\le N\}
$
into intervals 
\[
I_j= \left[F_j,F_{j+1}\right),
\]
for $j=1,\ldots,T$, where $F_j$ is the $j$th Fibonacci number.    We then compute the function $G$ by a Taylor expansion in the second variable, centered at $\pi F_j/d$,  and truncated to $B$ terms:
\[
 2\,\text{Re}(G(s,\pi n^2/d))\approx
\sum_{k=0}^B  G_{j,k}(t)(\pi/d)^k\cdot(n^2- F_j)^k,
\]
where
\begin{equation}\label{Eq:defgjk}
G_{j,k}(t)= 2\,\text{Re}(G^{(k)}(s,\pi F_j/d))/k!.
\end{equation}

\begin{theorem}
We can compute  $Z(t,\chi)$ as
\begin{equation}\label{Eq:thm}
Z(t,\chi)=
\sum_{j=1}^T\sum_{k=0}^{B} G_{j,k}(t)(\pi/d)^k\sum_{n^2\in
I_j}\chi(n)n^a( n^2-F_j)^k
\end{equation}
to $D$ digits of accuracy, where $T$ and $B$ are both $O(\log(d))$, the implied constants depending on $D$.
\end{theorem}

The expression
\begin{equation}\label{Eq:pre}
C_{jk}\overset{\text{def.}}=\sum_{n^2\in I_j}\chi(n)n^a( n^2-F_j)^k
\end{equation}
is a precomputation independent of $s$ in \emph{integers} which is $O(N\cdot B)= O(d^{1/2}\log(d)^2)$.
Subsequently, individual evaluations of $Z(t,\chi)$  cost only $O(T\cdot B)= O(\log(d)^2)$.

\begin{proof}
Of course, the outermost sum on $j\le T$ and the innermost sum on $n^2$ in $I_j$ combine to give the squares of all $n\le N$; the middle sum giving the needed Taylor expansions.
We need $N^2$ to be in the last interval $I_T$, so
\[
N^2<F_{T+1}\approx \Phi^{T+1}/\sqrt{5},\quad\text{where}\quad\Phi=\frac{1+\sqrt{5}}{2};
\]
with $N\ll d^{1/2+\epsilon}$ by (\ref{Eq:crude}), this implies that $T\ll \log(d)$ suffices.  

This is all well and good, but we need to show that using Taylor expansions at points spaced in what is essentially a geometric progression, does not require an unreasonable number of terms $B$ in each expansion in order to compute accurately.
Use $|\chi(n)|\le1$, $n^a\le  n$, and the rough estimate $d^{1/2}$ for the $L$-series truncation parameter $N$.  Assuming the errors we make in computing each
$
G(s,\pi n^2/d)
$ are independent with standard deviation $\epsilon$, then the standard error in the sum (\ref{Eq:1}) is bounded by \cite{D}
\[
\left(\sum_{n=1}^{d^{1/2}}(n\epsilon)^2\right)^{1/2}\ll \epsilon \cdot d^{3/4},
\]
where we approximated a sum by an integral.  We want $\epsilon \cdot d^{3/4}<10^{-D}$, or
\[
\epsilon< 10^{-D}d^{-3/4},
\]
which will determine how many terms $B$ we need in each Taylor expansion.  We'll use the weaker inequality in the Lemma with
\[
x_0=\pi F_j/d,\quad x<\pi F_{j+1}/d,
\]
which makes the error
\[
\epsilon<\frac{\left(\frac{F_{j+1}}{F_j}-1\right)^B}{B}<\left(\Phi-1\right)^B.
\]
Thus we want
\[
\left(\Phi-1\right)^B<10^{-D}d^{-3/4},\quad \text{or}\quad 10^D d^{3/4}<\left(\Phi-1\right)^{-B}=\left(\frac{\sqrt{5}-1}{2}\right)^B,
\]
and so $B=O(\log(d))$ suffices.
\end{proof}

For a single function evaluation (for example, determining whether $Z(0.,\chi)>0.$) this algorithm is no improvement over \cite{W}; summing the series requires $O(d\log(d))^{1/2}$ terms by (\ref{Eq:crude}).  If one wants to do an arbitrarily large number of function evaluations, the improvement is spectacular:  from exponential down to polynomial (in terms of the number of digits of $d$ which is $\approx \log(d)$).  This is deceptive, though, because what one really wants to do is find the all zeros with, say, $0\le t<1$.  (Larger $t$ intervals requires computing $G(s,x)$ via the methods of \cite{Rub} which in turn necessitates re-doing the precomputation.)  Since there are $O(\log(d))$ such zeros and each can be found with $O(1)$ evaluations, the precomputation still dominates as a theoretical result.  But as Jan L. A. van de Snepscheut\footnote{not Yogi Berra.} wrote \begin{quote} \lq\lq In theory, there is no difference between theory and practice. But, in practice, there is.\rq\rq\end{quote}

\section{Data}

\begin{figure}
\centering
\includegraphics[scale=1, viewport=0 10 360 250,clip]{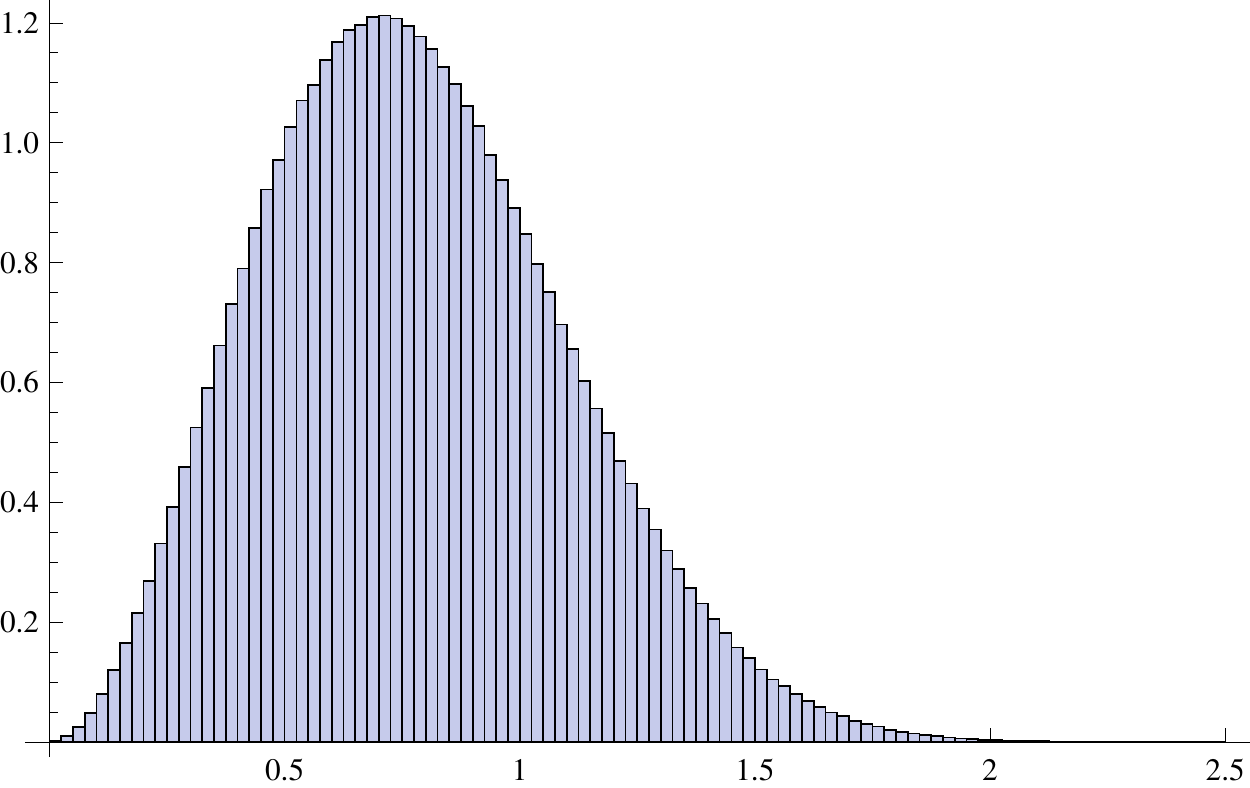}
\caption{Histogram of the lowest zero.}\label{F:first}
\end{figure}

\begin{sidewaysfigure}
\includegraphics[scale=1.5, viewport=0 0 350 250,clip]{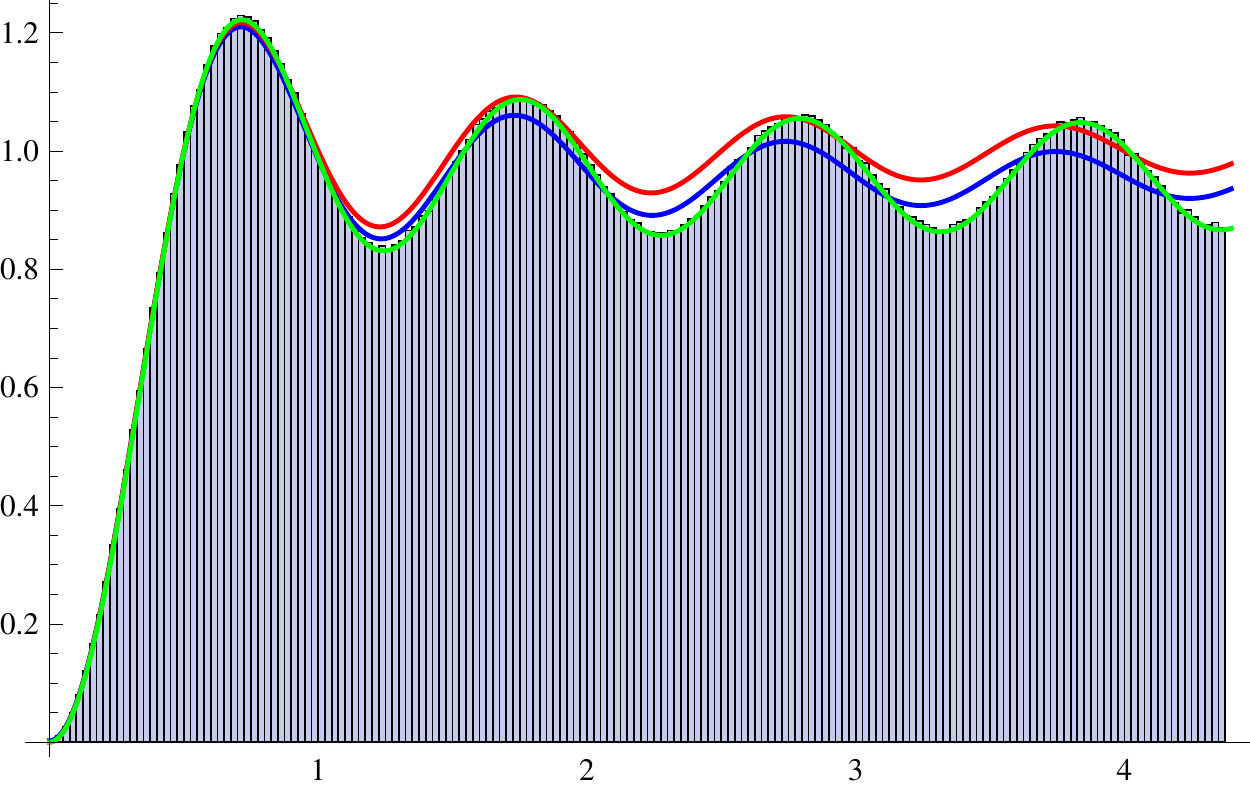}
\caption{Histogram of all zeros.}\label{F:all}
\end{sidewaysfigure}

Zeros with $t<1.$ were computed for all of the $1,\!039,\!654$  negative fundamental discriminants $-d$ in the range $10^{12}\le d\le 10^{12}+ 10^7$, a total of $12,\!202,\!567$ zeros.  Figure \ref{F:first} shows a histogram for the imaginary part of the lowest lying zero, rescaled by $\log(10^{12})/(2\pi)$.   The lowest zero found was at $t=0.0013104755$ corresponding to  the discriminant $-1,\!000,\!008,\!582,\!815$.  

Figure \ref{F:all} shows the histogram of imaginary parts of all the zeros, again rescaled by $\log(10^{12})/(2\pi)\sim 4.39761$.  The upper curve (in red) is the main term $1-\sin(2\pi \tau)/(2\pi \tau)$ for the symplectic random matrix model for the 1-level density.  The lower curve (in blue) includes also terms from (\ref{Eq:explicit}) which are $O(1/\log(X))$.  (In the notation of the Appendix, $X=10^{12}$ and $\Delta X= 10^7$.)  This version is derived from the Explicit Formula.  The fit is visibly poor for these values of $X$ and $\Delta X$.

As usual,  we assumed in (\ref{Eq:0}) that $\supp(\hg) \subset (-\sigma, \sigma) \subset (-1,1)$, so that
\[
 \intii g(\tau) \left(-\frac{\sin(2\pi \tau)}{2\pi \tau}\right)\, d\tau = -g(0)/2.
\]
It is really the $-g(0)/2$ term which appears in the proof via the Explicit Formula.  Miller \cite{Miller} derives a version of the 1-level density in which the term $-g(0)/2$ is replaced by a more complicated expression, see (\ref{Eq:mess}) in the Appendix.  This version is shown in green in Figure \ref{F:all}.  The fit appears to be very good.   Since (\ref{Eq:mess}) was first derived in \cite{CS1} from the $L$-functions Ratios Conjecture, the data seems to be good numerical evidence for the conjecture.
This also seems to indicate is that there is a lot of structure in the $O(1/\log(X))$ error in (\ref{Eq:0}), and the $L$-functions Ratios Conjecture captures that structure.

The data are available at
\begin{center}
\texttt{http://www.math.ucsb.edu/$\sim$stopple/quadratic.experiment}
\end{center}

\section{Implementation Notes}\label{S:IN}

\subsection*{Error estimates} With $D=15$ digits of accuracy and $d$ near $10^{12}$, the crude estimate (\ref{Eq:crude}) requires $N=5.4\cdot 10^6$ terms in the series.  We can actually do a little better.  Using this as a starting estimate, \emph{Mathematica}'s \texttt{FindRoot} uses a variant of the secant method to determine that $N=4.3\cdot 10^6$ satisfies (\ref{Eq:fine}), for a savings of better than $20 \% $.

The stronger inequality in the Lemma determines good values for the Taylor series truncation parameter $B$ in computing
$
G(s,\pi n^2/d)
$.
Consider the case $a=1$, i.e. a negative discriminant.
Assuming the errors (\ref{Eq:remainder}) in the terms are independent, and using $|\chi(n)|\le1$, $n^a= n$,   then the standard error in the sum over all $n^2$ in $I_j$ is bounded by \cite{D}
\[
\frac{\Gamma(B,\pi F_j/d)}{B!}\left(\sum_{F_j\le n^2<F_{j+1}}n^2(n^2/F_j-1)^{2B}\right)^{1/2}.
\]
We can estimate the sum by an integral
\[
\int_{\sqrt{F_j}}^{\sqrt{F_{j+1}}}t^2(t^2/F_j-1)^{2B}dt\approx F_j^{3/2}\int_1^{\Phi^{1/2}}u^2(u^2-1)^{2B}du.
\]
So the error from the sum over $n^2$ in $I_j$ is about
\begin{equation}\label{Eq:rhs}
\frac{\Gamma(B,\pi F_j/d)}{B!}F_j^{3/4}\left(\int_1^{\Phi^{1/2}}u^2(u^2-1)^{2B}du\right)^{1/2}
\end{equation}
For $d$ near $10^{12}$, we need $T= 65$ intervals, and it is easy to compute (\ref{Eq:rhs})  in \emph{Mathematica} for various $B$.  We see that $B=B(j)$ should increase linearly from $84$ at $j=31$ to $107$ at $j=65$, in order that the total of all errors is only about $10.^{-15}$.  (For $j<31$, the intervals $I_j$ contain not many more than $B(j)$ squares $n^2$, so the contribution of these $n$, namely $1\le n\le 1160$, is computed directly.)

The case $a=0$, i.e. positive discriminant is treated similarly.  It turns out one needs $B(j)$ to increase linearly from $70$ at $j=31$ to $78$ at $j=65$.

\subsection*{Algorithms} To find fundamental discriminants, I check the congruence condition, and test for divisibility by the squares of the first $200$ primes.  (The 94 examples divisible by the square of a prime larger than the 200th prime were easily identified with \emph{Mathematica} and removed from the data by hand.)

To compute $\Gamma(s)$ I use the Lanczos algorithm as in \cite{recipes,godfrey}.  Precomputed values of $\Gamma(s)$ allow efficient computation of incomplete Gamma functions $\Gamma(s,x)$ for various $x$ by the methods of \cite{recipes}: series expansion for $x<6.$ and continued fractions for $x\ge 6$.  These algorithms compare well with those implemented in \emph{Mathematica}, giving both absolute and relative error no worse than $10.^{-18}$ for the relevant range of $x$ and $|\text{Im}(s)|<1.$

To find zeros of $Z(t,\chi)$ the computation stepped through values in increments of $t$ of size $2\pi/\log(10^{12}/(2\pi))/50$, i.e. $1/50$th the mean gap between zeros.  When a sign change was observed, Ridder's method \cite{recipes} was used to find the root.  No effort was made to verify the GRH, or that all zeros of $Z(t,\chi)$ with $t<1.$ were located.  (However, the obvious check that $Z(0.,\chi)>0.$ was made.)

\subsection*{Hardware} Computations were done on a $3.0$ GHz $8$-core Mac Pro.  Both the integer arithmetic and also the recursion for the derivatives $G^{(k)}(s,x)$ were done with GMP 4.2.1 \cite{GMP} (ported to the Intel Core 2 Duo $64$ bit processor by Jason Worth Martin \cite{martin}.)\ \   For the rest of the floating point computations,  the C types \texttt{long double}, \texttt{long double complex} sufficed.  

\subsection*{Parallelization} Most of the computation consists of computing the values
$
(n^2-F_j)^k
$
in (\ref{Eq:pre}).  Since this is independent of $d$, there is a gain in efficiency by computing the quantities $C_{jk}$ in (\ref{Eq:pre}) for $8$ discriminants at a time.  Parallelism is easily implemented using Pthreads.  The contribution of the intervals $I_j$ is computed in $T$ separate threads for $8$ discriminants at a time.  Once all the precomputation is done, the zeros of $Z(t,\chi)$ for each of the $8$ characters $\chi$ are computed in $8$ separate threads.

\subsection*{Testing} Accuracy of computed zeros was tested three ways: first by recomputing well know examples \cite{Watkins1,Watkins2,W} of moderate sized  discriminants such as $-115,\!147$ and $-175,\!990,\!483$.  Second, I also implemented the method of \cite{W} directly in \emph{Mathematica} and compared a few examples for discriminants with absolute greater than $10^{12}$, with agreement to 15 digits.  Third, I compared with the unpublished data from Rubinstein's thesis \cite{Rub}.  This includes 3601 prime discriminants $-d$ with $10^{12}\le d\le 10^{12}+2\cdot 10^{5}$.  The data was in agreement with his to the 10 digits of accuracy he computed.  

\section{Appendix: Refined 1 Level Density}
This Appendix closely follows \cite{Miller} to determine the 1-level density, including lower order terms, for the family of quadratic Dirichlet $L$-functions.    Instead of considering the set of all fundamental $d<X$, I adapted the proof for 
\[
\mathcal{F}(X)=\{X<|d|<X+\Delta X\}
%\quad \text{with}\quad X^*\overset{\text{def.}}=\sharp \mathcal{F}(X).
\]
Where Miller treats the case when $\chi_d$ is an even function, i.e. $d>0$, I instead considered $\chi_d$ odd function, $-d<0$.
Throughout I assumed about $\Delta X$ that
\begin{equation}\label{hyp}
X^{1/2}\log(X)=o (\Delta X)\quad\text{and} \quad \Delta X=o(X).
\end{equation}
\begin{millertheorem}[Miller]  Let $g$ be an even Schwartz test
function such that $\supp(\hg) \subset
(-\sigma, \sigma)$, where $\hg$ denotes the Fourier transform of $g$.    Let
\[ 
A'(r) = \sum_p \frac{\log p}{(p+1)(p^{1+2r}-1)}. 
\]
Then
\begin{multline} \label{Eq:explicit}
\frac1{\sharp \mathcal{F}(X)}\sum_{d \in \mathcal{F}(X)} \sum_{\gamma_d}
g\left(\gamma_d \frac{\log X}{2\pi}\right)  =
 \intii
g(\tau)\left(1 -\frac{\sin(2\pi \tau)}{2\pi \tau}\right)d\tau+\\ 
   \frac1{\log
X} \intii g(\tau) \Bigg[-\log(\pi)+\tre
\frac{\Gamma'}{\Gamma}\left(\frac34+\frac{i\pi \tau}{\log X}\right)+\\2\tre\frac{\zeta'}{\zeta}\left(1+
\frac{4\pi i\tau}{\log X}\right) + 2\tre A'\left(\frac{2\pi i \tau}{\log X}\right) \Bigg] d\tau \\ 
+o\left(\frac{1}{\log X}\right) +O\left(\frac{X^{\sigma/2} \log^6 X}{\Delta X^{1/2}}\right). 
\end{multline} 
\end{millertheorem}
Of course, to get the $O(X^{\sigma/2} \log^6 X/\Delta X^{1/2})$ error to be $o(1/\log X)$ we would need to restrict the support of $\hg$ to be $\subset (-1/2,1/2)$.

%For my notes on Miller's proof, get the TeX source of this paper from the arXiv and remove the comment delimiters: \begin{verbatim}\begin{comment}
%\end{comment} \end{verbatim}

Figure \ref{F:three} shows each of the three non-constant terms which are absorbed in the  $O(1/\log(X))$ error in (\ref{Eq:0}), all on the same scale of $2\pi\tau/\log(X)=t$. 
The $\Gamma^\prime/\Gamma$ term (in green) is slowly increasing and very smooth, while $A^\prime$ (in blue) is small and wobbly.
Observe that when $\zeta(1/2+i\gamma)=0$, the contribution at $t=\gamma/2$ of $\zeta^\prime/\zeta(1+2it)$ (in red)  is positive and large.
This follows from \cite[Theorem 9.6(A)]{Tit}, which says that
\[
\frac{\zeta^\prime(s)}{\zeta(s)}=\sum_{|t-\gamma|\le1}\frac{1}{s-\rho}+O(\log(t)),
\]
so up to a small error, the logarithmic derivative is determined by the nearby zeros $\rho$.  This \lq resurgence\rq\ of the zeros of $\zeta(s)$ does not play much role in the data ($t<1.$) presented here.

\begin{figure*}
\begin{center}
\includegraphics[scale=1, viewport=0 0 300 200,clip]{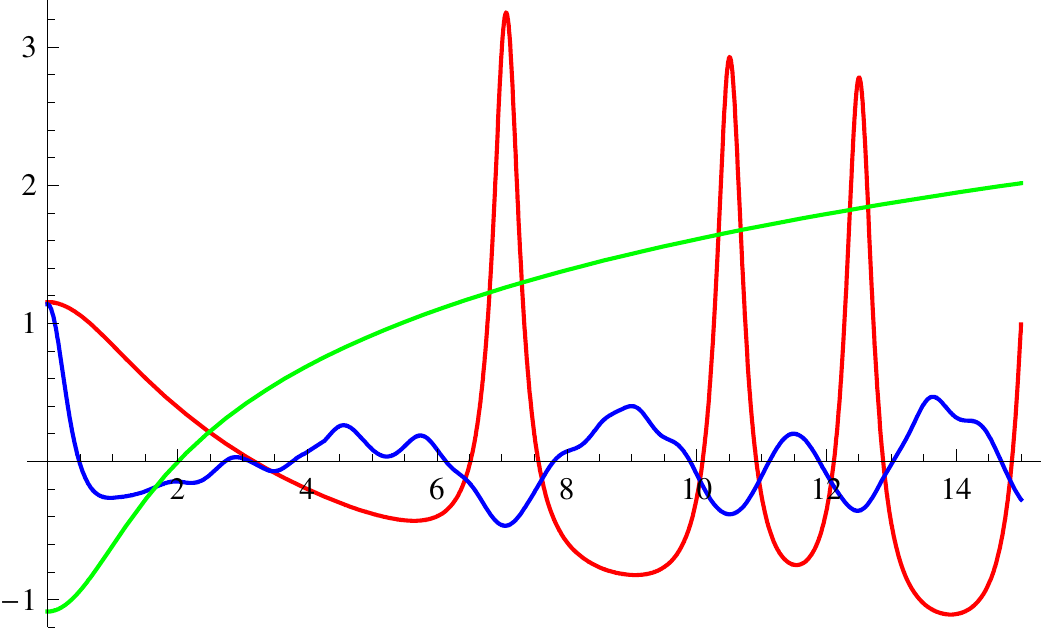}
\caption{All three non constant $O(1/\log(X))$ terms.  The $\zeta^\prime/\zeta$ term is in red, the $\Gamma^\prime/\Gamma$ term in green, and the $A^\prime$ term in blue.}\label{F:three}
\end{center}
\end{figure*}

Since the fit of the data to even the \lq refined\rq\ 1-level density is poor, we turn instead to the prediction inspired by the $L$-functions Ratios Conjecture.  The term $-\sin(2\pi \tau)/(2\pi\tau)$ is replaced by the real part of
\begin{multline}\label{Eq:mess}
R(\tau,X)=
\frac{-2}{\sharp \mathcal{F}(X)\log X}  \sum_{d \in \mathcal{F}(X)}\exp\left(-2\pi i
\tau \frac{\log(d/\pi)}{\log X}\right) \times\\
\frac{\Gamma\left(\frac34-\frac{\pi
i \tau}{\log X}\right)}{\Gamma\left(\frac34+\frac{\pi i \tau}{\log
X}\right)}\frac{\zeta(2)
 \zeta\left(1 - \frac{4\pi i \tau}{\log
X}\right)}{\zeta\left(-2-\frac{4\pi i
\tau}{\log X}\right)}.
\end{multline}
(We have simplified the notation from \cite[(1.6)]{Miller}; see also his Lemma 2.4).  Miller shows \cite[Lemma 2.1]{Miller} that on the Riemann Hypothesis,
\[
\intii g(\tau)R(\tau,X)d\tau=-g(0)/2+O(X^{-3/4(1-\sigma)+\epsilon}),
\]
and unconditionally with a larger error.  Here as usual $\supp(\hg) \subset
(-\sigma, \sigma)$.

In order that the prediction not depend on the specific discriminants in $\mathcal F(X)$, we use summation by parts \cite[Remark 2.3]{Miller} to estimate
\begin{multline*}
\sum_{d<X}\exp\left(-2\pi i \tau \frac{\log(d/\pi)}{\log X}\right)=\\
\frac{3 X}{\pi^2}\left(\frac{X}{\pi}\right)^{-2\pi i\tau/\log(X)}\frac{1}{1-2\pi i\tau/\log(X)}+O(X^{1/2})
\end{multline*}
and similarly with the sum over $d<X+\Delta X$.  The difference of these, divided by $\sharp \mathcal{F}(X)=3\Delta X/\pi^2 +O(X+\Delta X)^{1/2}$ is used in an estimate of (\ref{Eq:mess}) and denoted $R_{\text{est}}(\tau,X)$.  Figure \ref{F:gzero} shows how $R_{\text{est}}(\tau,X)$ (in red)  compares to $-\sin(2\pi \tau)/(2\pi\tau)$ (blue).

The graph in green in Figure \ref{F:all} has $-\sin(2\pi \tau)/(2\pi\tau)$ replaced by $R_{\text{est}}(X)$, and also includes the other $O(1/\log(X))$ terms.  

\begin{figure*}
\begin{center}
\includegraphics[scale=1, viewport=0 0 300 200,clip]{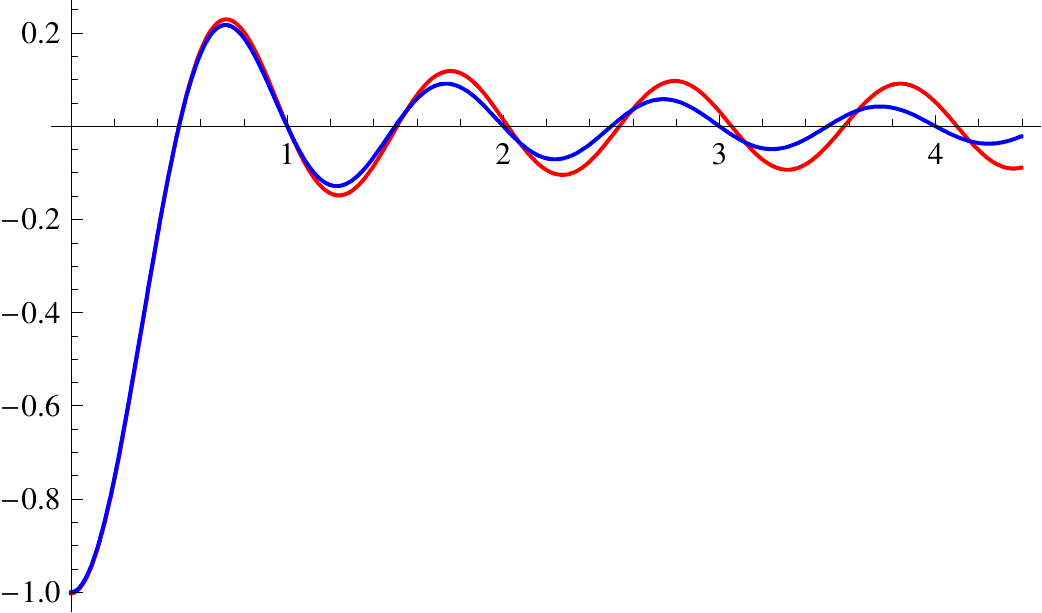}
\caption{$-\sin(2\pi \tau)/(2\pi \tau) $ (in blue) v. $R_{\text{est}}(\tau,X)$ (in red)}\label{F:gzero}
\end{center}
\end{figure*}

\end{document}